\documentclass[reqno, 11pt]{amsart}
\usepackage{amsmath,amsthm,amssymb}
\usepackage{graphicx}
\usepackage{hyperref} 
\hypersetup{
	setpagesize=false,
	bookmarksnumbered=true,%
	bookmarksopen=true,%
	colorlinks=true,%
	linkcolor=magenta,
	citecolor=blue,
}
\theoremstyle{plain}
\newtheorem{thm}{Theorem}[section]
\newtheorem{lem}[thm]{Lemma}

\newtheorem{cor}[thm]{Corollary}
\theoremstyle{definition}
\newtheorem{dfn}[thm]{Definition}
\newtheorem{rem}[thm]{Remark}
\newtheorem{qst}[thm]{Question}
\newcommand{\CC}{\mathbb{C}}
\newcommand{\calT}{\mathcal{T}}
\newcommand{\calD}{\mathcal{D}}
\newcommand{\Int}{\mathrm{Int\,}}

\DeclareMathOperator{\id}{id}

\title[Non-diffeomorphic minimal genus relative trisections]{Non-diffeomorphic minimal genus relative trisections of the same 4-manifold}
\author[Natsuya Takahashi]{Natsuya Takahashi}
\date{June 5, 2024.}
\address{Department of Pure and Applied Mathematics, Graduate School of Information Science and Technology, Osaka University, 1-5 Yamadaoka, Suita, Osaka 565-0871, Japan}
\email{nt-takahashi@ist.osaka-u.ac.jp}
\begin{document}
\begin{abstract}
We show the existence of a $4$-manifold with boundary that admits two non-diffeomorphic minimal genus relative trisections of the same $(g,k;p,b)$-type. To prove this, we introduce a simple operation that produces a trisection diagram of a closed $4$-manifold from a relative trisection diagram.
\end{abstract}
\maketitle
%
\section{Introduction}

Any compact smooth $4$-manifold $X$ admits a \textit{trisection}, which is a decomposition of $X$ into three $4$-dimensional $1$-handlebodies intersecting along a compact surface.
A trisection for a $4$-manifold with non-empty boundary is called a \textit{relative trisection}.
The notion of (relative) trisections was introduced by Gay and Kirby~\cite{GayKir16} as a $4$-dimensional analogue of Heegaard splittings for $3$-manifolds.
In trisection theory, it is important to study the similarities between Heegaard splittings and trisections.

The classification problems for Heegaard splittings have been studied based on appropriate equivalence relations, such as homeomorphism and isotopy.
For example, Waldhausen~\cite{Wald68} proved that all Heegaard splittings of the $3$-sphere with the same genus are isotopic.
%
%
Similar statements hold for other $3$-manifolds, such as lens spaces (\cite{BonOta83}, \cite{HodRub85}) and $\#^n(S^1 \times S^2)$ (\cite{Wald68}, see also \cite{CarOer05}).
%
%
%
On the other hand, it is known that \textit{there exist $3$-manifolds admitting more than one homeomorphism type (or isotopy type) of minimal genus Heegaard splittings} (see \cite{Eng70}, \cite{Bir75}, \cite{BirGonMon76}, for example).
%
%
%
%
In this paper, we study the $4$-dimensional version of this fact.
Two (relative) trisections $X=X_1\cup X_2\cup X_3$ and $X'=X'_1\cup X'_2\cup X'_3$ are called \textit{diffeomorphic} if there exists an orientation-preserving diffeomorphism $f:X\to X'$ such that $f(X_i)=X'_i$ for each $i\in\{1,2,3\}$.
The following question asks about the existence of non-equivalent trisections of the same $4$-manifold.

\begin{qst}\label{qst:distinct_tris}
Does there exist a $4$-manifold admitting more than one diffeomorphism type of minimal genus (relative) trisections?
\end{qst}

In fact, Islambouli \cite{Isl21} has already shown that the answer is yes for closed $4$-manifolds.
%
He proved the existence of infinitely many spun Seifert fiber spaces that admit non-diffeomorphic trisections of the same $(g,k)$-type.
(In addition, Isoshima and Ogawa~\cite{IsoOga24a} also showed that the answers to the bisectional and 4-sectional analogues of Question~\ref{qst:distinct_tris} are yes.)
%
In \cite{Isl21} and \cite{IsoOga24a}, they used the notion of Nielsen equivalence to prove the non-equivalence of trisections and multisections.


On the other hand, the relative version of Question~\ref{qst:distinct_tris} has not been sufficiently studied yet.
Note that when considering Question~\ref{qst:distinct_tris} for $4$-manifolds with boundary, we exclude those with $S^3$ boundary.
This is because if a closed $4$-manifold $X$ admits two non-diffeomorphic $(g,k)$-trisections, then $X-\Int{D^4}$ also admits two non-diffeomorphic $(g,k;0,1)$-relative trisections.
(For details see Remark~\ref{rem:equiv-rtd-X-D4}.)
The main result of this paper is to prove that the answer to Question~\ref{qst:distinct_tris} is yes for $4$-manifolds with boundary by a completely different method from \cite{Isl21} and \cite{IsoOga24a}.

\begin{thm}\label{main:distinct-reltris}
There exists a $4$-manifold with boundary ($\not\cong S^3$) that admits two non-diffeomorphic minimal genus relative trisections of the same $(g,k;p,b)$-type, and furthermore, the induced open books are equivalent.
\end{thm}

We prove this theorem by giving two non-diffeomorphic relative trisections for $S^2\times D^2$.
In Section~\ref{sect:proof_nonequiv}, we will show that two $(2,1;0,2)$-relative trisection diagrams of $S^2\times D^2$ of Figure~\ref{fig:td-S2xD2_2} are not diffeomorphism and handleslide equivalent.
Moreover, the open book decompositions on the boundary $S^2\times S^1$ induced by these relative trisections are equivalent, so they cannot be distinguished by the boundary data.

To prove Theorem~\ref{main:distinct-reltris}, we introduce a simple operation that produces a trisection diagram of a closed $4$-manifold from a relative trisection diagram.
For a $(g,k;0,b)$-relative trisection diagram $\calD$, our operation yields a $(g,k-b+1)$-trisection diagram $\widehat{\calD}$ by gluing $b$ disks to the boundaries of the trisection surface (see Figure~\ref{fig:gd}).
For a precise description of this operation, see Section~\ref{sect:gd}.
We will prove the non-equivalence of relative trisection diagrams by using the following theorem.

\begin{thm}\label{thm:equiv-rtd-gd}
If two $(g,k;0,b)$-relative trisection diagrams $\calD$ and $\calD'$ are diffeomorphism and handleslide equivalent, then the $(g,k-b+1)$-trisection diagrams $\widehat{\calD}$ and $\widehat{\calD}'$ are also diffeomorphism and handleslide equivalent.
\end{thm}

A brief outline of the proof of Theorem~\ref{main:distinct-reltris} is as follows.
If we assume that the two relative trisection diagrams of $S^2\times D^2$ shown in Figure~\ref{fig:td-S2xD2_2} are diffeomorphism and handleslide equivalent, then we obtain a contradiction.
The induced trisection diagrams, which should be diffeomorphism and handleslide equivalent, yield distinct closed $4$-manifolds.

Our method of distinguishing relative trisections is also applicable to other examples.
In \cite{tak23a}, the author gave two minimal genus-$3$ relative trisections of the Akbulut cork, which is a contractible $4$-manifold with boundary.
In a forthcoming paper, we will apply our method to show that these relative trisections are not diffeomorphic.
Moreover, by considering blow-ups of the Akbulut cork, we will also prove that, for every integer $n\geq2$, there exists a $4$-manifold with boundary ($\not\cong S^3$) that admits non-diffeomorphic relative trisections of minimal genus $n$.

\begin{figure}[!tbp]
\centering
\includegraphics[scale=0.85]{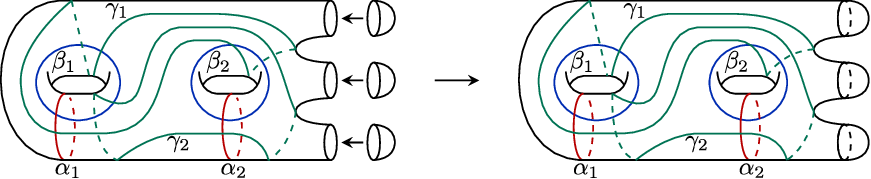}
\caption{An operation that produces a trisection diagram $\widehat{\calD}$ from a relative trisection diagram $\calD$.}
\label{fig:gd}
\end{figure}

\subsection*{Conventions}

Throughout this paper, all manifolds and surfaces are assumed to be compact, connected, oriented, and smooth.
In addition, if two smooth manifolds $X$ and $Y$ are orientation-preserving diffeomorphic, then we write $X\cong Y$. 
Let $\Sigma_{g}$ be a closed, connected, oriented surface of genus $g$, and let $\Sigma_{g,b}$ be a compact, connected, oriented surface of genus $g$ with $b$ boundary components.

\section{Preliminaries}\label{sec:prel}

In this section, we will see several properties of trisections.
A $(g,k)$-trisection $X=X_1\cup X_2\cup X_3$ is a decomposition of a closed $4$-manifold $X$ into three genus-$k$ $4$-dimensional $1$-handlebodies $X_1$, $X_2$, and $X_3$, meeting pairwise in $3$–dimensional $1$–handlebodies, with triple intersection $X_1\cap X_2\cap X_3$ a closed surface of genus $g$.
Generally, the term ``trisection'' refers to the above decomposition for a closed $4$-manifold.
In this paper, we also focus on trisections for $4$-manifolds with boundary, called relative trisections.
For a compact, connected, oriented, smooth $4$-manifold $X$ with non-empty connected boundary, if a decomposition $X = X_1\cup X_2\cup X_3$ is a $(g, k; p, b)$-relative trisection, then the followings hold:
\begin{itemize}
\item
The four integers $g$, $k$, $p$, and $b$ satisfy the inequalities $g,k,p\geq0$, $b\geq1$, and $2p+b-1\leq k\leq g+p+b-1$.
\item
The sectors $X_1$, $X_2$ and $X_3$ are diffeomorphic to the $4$-dimensional $1$-handlebody of genus $k$.
\item
The double intersection $X_i\cap X_j (=\partial{X_i}\cap\partial{X_j})$ is a compression body between $\Sigma_{g,b}$ to $\Sigma_{p,b}$ that is diffeomorphic to the $3$-dimensional $1$-handlebody of genus $g+p+b-1$.
\item
The triple intersection $X_1\cap X_2\cap X_3$ is diffeomorphic to $\Sigma_{g,b}$.
\item
There exists an open book decomposition on $\partial{X}$ with pages $\Sigma_{p,b}$.
\end{itemize}

For the precise definitions of trisections and relative trisections, see \cite{GayKir16} and \cite{CasGayPin18_1}.
We sometimes write a (relative) trisection by the symbol $\calT$.
For a (relative) trisection $X = X_1\cup X_2\cup X_3$, the genus of the triple intersection surface $X_1\cap X_2\cap X_3$ is called the \textit{genus} of the trisection.
The \textit{trisection genus} of a smooth $4$-manifold $X$ is defined as the minimal integer $g$ such that $X$ admits a (relative) trisection of genus $g$.
A (relative) trisection $\calT$ of a $4$-manifold $X$ is \textit{minimal genus} if the genus of $\calT$ is equal to the trisection genus of $X$.

The notions of equivalence for trisections are given by the following.

\begin{dfn}
Two (relative) trisections $X=X_1\cup X_2\cup X_3$ and $X'=X'_1\cup X'_2\cup X'_3$ are called diffeomorphic, if there exists a diffeomorphism $f:X\to X'$ such that $f(X_i)=X'_i$ for each $i\in\{1,2,3\}$.
The above two trisections are called \textit{isotopic}, if there exists an ambient isotopy $\{f_{t}:X\to X'\}_{t\in[0,1]}$ such that $f_{1}(X_i)=X'_i$ for each $i\in\{1,2,3\}$.
\end{dfn}

Remark that if two trisections are isotopic, then they are diffeomorphic.
We easily see that diffeomorphic (relative) trisections have the same $(g,k)$-type (or $(g,k;p,b)$-type).

Now, we will introduce the definition of trisection diagrams.

\begin{dfn}
Let $\Sigma$ and $\Sigma'$ be compact, connected, oriented surfaces.
For $i\in \{1,\ldots,n\}$, let $\alpha^i$ and $\beta^i$ be families of $k$ pairwise disjoint simple closed curves on $\Sigma$ and $\Sigma'$, respectively.
Two $(n+1)$-tuples $(\Sigma;\alpha^1,\ldots,\alpha^n)$ and $(\Sigma';\beta^1,\ldots,\beta^n)$ are \textit{diffeomorphism and handleslide equivalent} if they are related by diffeomorphisms of $\Sigma$ and handleslides within each $\alpha^i$ (i.e., we are only allowed to slide curves from $\alpha^i$ over other curves from $\alpha^i$, but not over curves from $\alpha^j$ when $j\neq i$).
\end{dfn}

The followings are the definitions of a (relative) trisection diagram.

\begin{dfn}\label{def:rtd}
Let $g$ and $k$ be integers satisfying $0\leq k\leq g$.
Let $\Sigma$ be a closed surface which is diffeomorphic to $\Sigma_{g}$, and let $\alpha$, $\beta$, and $\gamma$ be families of $g$ pairwise disjoint simple closed curves on $\Sigma$. 
A $4$-tuple $(\Sigma;\alpha,\beta,\gamma)$ is called a $(g,k)$-\textit{trisection diagram} if $(\Sigma;\alpha,\beta)$, $(\Sigma;\beta,\gamma)$, and $(\Sigma;\gamma,\alpha)$ are diffeomorphism and handleslide equivalent to the standard diagram $(\Sigma_{g};\delta,\epsilon)$ of type $(g,k)$ shown in Figure \ref{fig:std-td}.
\end{dfn}
\begin{figure}[!htbp]
\centering
\includegraphics[scale=1]{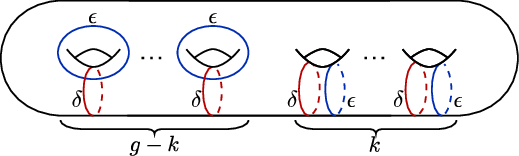}
\caption{The standard diagram $(\Sigma_{g};\delta,\epsilon)$ of type $(g,k)$, where the red curves are $\delta$ and the blue curves are $\epsilon$.}
\label{fig:std-td}
\end{figure}

\begin{dfn}\label{def:rtd}
Let $g$, $k$, $p$, and $b$ be integers satisfying $g,k,p\geq0$, $b\geq1$, and $2p+b-1\leq k\leq g+p+b-1$.
Let $\Sigma$ be a compact surface which is diffeomorphic to $\Sigma_{g,b}$, and let $\alpha$, $\beta$, and $\gamma$ be families of $g-p$ pairwise disjoint simple closed curves on $\Sigma$. 
A $4$-tuple $(\Sigma;\alpha,\beta,\gamma)$ is called a $(g,k;p,b)$-\textit{relative trisection diagram} if $(\Sigma;\alpha,\beta)$, $(\Sigma;\beta,\gamma)$, and $(\Sigma;\gamma,\alpha)$ are diffeomorphism and handleslide equivalent to the standard diagram $(\Sigma_{g,b};\delta,\epsilon)$ of type $(g,k;p,b)$ shown in Figure \ref{fig:td-stdgkpb}.
\end{dfn}
\begin{figure}[!tbp]
\centering
\includegraphics[scale=1]{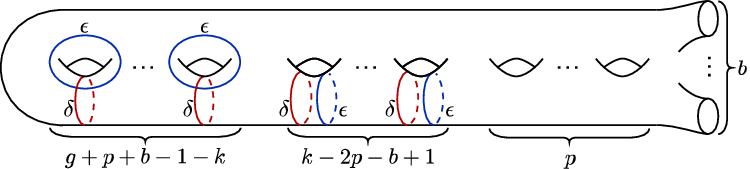}
\caption{The standard diagram $(\Sigma_{g,b};\delta,\epsilon)$ of type $(g,k;p,b)$, where the red curves are $\delta$ and the blue curves are $\epsilon$.}
\label{fig:td-stdgkpb}
\end{figure}

We sometimes denote a (relative) trisection diagram by the symbol $\calD$.
For a trisection diagram $\calD=(\Sigma;\alpha,\beta,\gamma)$, we draw $\alpha$, $\beta$, and $\gamma$ curves by red, blue, and green curves, respectively.
For a diagram $(\Sigma;\alpha^1,\ldots,\alpha^n)$ (not necessarily a trisection diagram) and a diffeomorphism $f$ on $\Sigma$, let $f(\Sigma;\alpha^1,\ldots,\alpha^n):=(f(\Sigma);f(\alpha^1),\ldots,f(\alpha^n))$, where each $f(\alpha^i)$ is the family of the images of the curves under $f$.

As seen above, a trisection diagram is defined independently of a trisection, however, there is a natural one-to-one correspondence between them.

\begin{thm}[Gay--Kirby~\cite{GayKir16}, Castro--Gay--Pinz\'{o}n-Caiced~{\cite{CasGayPin18_1}}]\label{thm:rt-rtd}
 The following (i), (ii), and (iii) hold.
\begin{enumerate}
\item
 For any $(g,k)$- (or $(g,k;p,b)$-relative) trisection diagram $(\Sigma;\alpha,\beta,\gamma)$, there exists a unique (up to diffeomorphism) trisected $4$-manifold $X=X_1\cup X_2\cup X_3$ satisfying the following conditions.
\begin{itemize}
\item
$X_1\cap X_2\cap X_3\cong \Sigma$.
\item
Under the above identification, each of $\alpha$, $\beta$, and $\gamma$ curves bound compressing disks of $X_3\cap X_1$, $X_1\cap X_2$, $X_2\cap X_3$, respectively.
\end{itemize}
\item
For any (relative) trisection $\calT$, there exists a (relative) trisection diagram $\calD$ such that $\calT$ is induced from $\calD$ by (i).
\item
Let $\calD_1$ and $\calD_2$ be (relative) trisection diagrams.
If the induced (relative) trisections $\calT_{\calD_1}$ and $\calT_{\calD_2}$ are diffeomorphic, then $\calD$ and $\calD'$ are diffeomorphism and handleslide equivalent.
\end{enumerate}
\end{thm}

Theorem~\ref{thm:rt-rtd} means the following one-to-one correspondence.
\begin{equation*}
\frac{\{\text{(relative) trisections}\}}{\text{diffeomorphism}} \quad {\stackrel{1:1}{\longleftrightarrow}} \quad \frac{\{\text{(relative) trisection diagrams}\}}{\text{diffeomorphism and handleslide equivalent}}.
\end{equation*}

For a (relative) trisection diagram $\calD$, we denote the induced (relative) trisection by $\calT_\calD$, and we also denote the induced trisected $4$-manifold by $X_{\calD}$.

\section{An operation that produces a trisection diagram from a relative trisection diagram}\label{sect:gd}

We introduce a simple operation that produces a trisection diagram of a closed $4$-manifold from a relative trisection diagram.

\begin{lem}\label{lem:gd_welldef}
Let $\calD=(\Sigma;\alpha,\beta,\gamma)$ be a $(g,k;0,b)$-relative trisection diagram.
We let $\widehat{\Sigma}$ be a genus-$g$ closed surface obtained from $\Sigma$ by gluing $b$ disks $D_1,\ldots,D_b$ to the boundary circles (i.e., $\widehat{\Sigma}:=\Sigma\cup(\sqcup^{b}_{i=1}D_i)$).
(Now we can consider that $\alpha$, $\beta$, and $\gamma$ are families of curves on the resulting closed surface $\widehat{\Sigma}$.)
Then, $\widehat{\calD}:=(\widehat{\Sigma};\alpha,\beta,\gamma)$ is a $(g,k-b+1)$-trisection diagram.
\end{lem}

\begin{proof}
We need to verify that the three triples $(\widehat{\Sigma};\alpha,\beta)$, $(\widehat{\Sigma};\beta,\gamma)$, and $(\widehat{\Sigma};\gamma,\alpha)$ are diffeomorphism and handleslide equivalent to the standard diagram of type $(g,k-b+1)$.
However, we will prove the case of $(\widehat{\Sigma};\alpha,\beta)$ only, because the other cases can be proved by a similar argument.
Since $\calD$ satisfies the definition of a relative trisection diagram, $(\Sigma;\alpha,\beta)$ is diffeomorphism and handleslide equivalent to the standard diagram of type $(g,k;0,b)$, that is, there exists a diffeomorphism $f:\Sigma\to\Sigma_{g,b}$ and a sequence of handleslides $s=(s_1,s_2,\ldots,s_n)$ such that $f(\Sigma;\alpha,\beta)=(\Sigma_{g,b};f(\alpha),f(\beta))$ becomes isotopic to the standard diagram of Figure~\ref{fig:std-gk0b} after performing the handleslides $s$.

We will now extend $f:\Sigma\to\Sigma_{g,b}$ to a diffeomorphism $\widehat{f}:\widehat{\Sigma} \to \Sigma_g$.
Suppose that $\Sigma_{g,b}$ is canonically embedded in $\Sigma_{g}$ as shown in Figure~\ref{fig:embedded_surf}.
For each $i\in\{1,2,\ldots,b\}$, let $g_i:\partial{D_i}\to\partial{\Sigma}$ be a gluing map of each disk $D_i$.
Then, the map $f\circ g_i:\partial{D_i}\to\partial{\Sigma_{g,b}}\subset\Sigma_{g}$ is a diffeomorphism on a circle.
By the theorem of Smale~\cite{Sma59}, any diffeomorphism on the circle $S^1=\partial{D^2}$ extends to the disk $D^2$.
Thus, there exists a diffeomorphism $\phi_i:D_i\to D^2$ such that $\phi_{i}\mid_{\partial{D_i}}=f\circ g_i$.
Then, define the diffeomorphism $\widehat{f}:\widehat{\Sigma}\to\Sigma_{g}$ by $\widehat{f}(x):=f(x)$ if $x\in\Sigma$ and $\widehat{f}(x):=\phi_{i}(x)$ if $x\in D_i$.
Note that $\widehat{f}$ is well-defined.
Then, $\widehat{f}(\widehat{\Sigma};\alpha,\beta)=(\Sigma_{g};f(\alpha),f(\beta))$ becomes isotopic to the standard diagram of type $(g,k-b+1)$ after performing the handleslides $s=(s_1,s_2,\ldots,s_n)$.
Here, we can consider that each handleslide $s_i$ is performed on $\Sigma_g$, since $\Sigma_{g,b}$ is embedded in $\Sigma_{g}$.
\end{proof}
\begin{figure}[!htbp]
\centering
\includegraphics[scale=1]{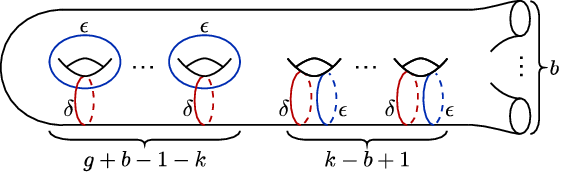}
\caption{The standard diagram $(\Sigma_{g,b};\delta,\epsilon)$ of type $(g,k;0,b)$.}
\label{fig:std-gk0b}
\end{figure}
\begin{figure}[!htbp]
\centering
\includegraphics[scale=0.7]{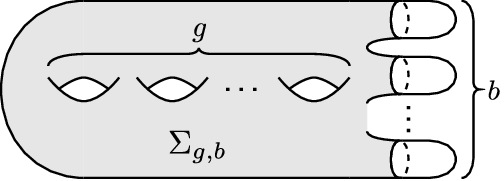}
\caption{A canonical embedding of $\Sigma_{g,b}$ into $\Sigma_g$.}
\label{fig:embedded_surf}
\end{figure}
\begin{figure}[!htbp]
\centering
\includegraphics[scale=0.9]{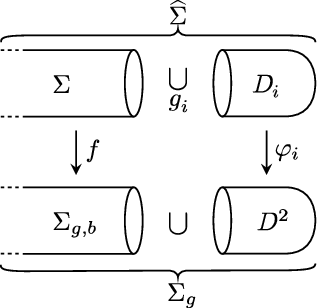}
\caption{A construction of $\widehat{f}:\widehat{\Sigma}\to\Sigma_g$.}
\label{fig:fhat_gdwelldef}
\end{figure}

\begin{proof}[Proof of Theorem~\ref{thm:equiv-rtd-gd}]
The proof is similar to that of Lemma~\ref{lem:gd_welldef}.
Let $\calD:=(\Sigma;\alpha,\beta,\gamma)$ and $\calD':=(\Sigma';\alpha',\beta',\gamma')$.
Since $\calD$ and $\calD'$ are diffeomorphism and handleslide equivalent, there exists a diffeomorphism $f:\Sigma\to\Sigma'$ and a sequence of handleslides $s=(s_1,s_2,\ldots,s_n)$ such that $f(\calD)=(\Sigma';f(\alpha),f(\beta),f(\gamma))$ becomes isotopic to $\calD'=(\Sigma';\alpha',\beta',\gamma')$ after performing the handleslides $s$.

We will now extend the diffeomorphism $f$ over $\widehat{\Sigma}=\Sigma\cup(\sqcup_{b}D^2_i)$.
For each $i\in\{1,2,\ldots,b\}$, let $g_i:\partial{D_i}\to\partial{\Sigma}$ and $g'_i:\partial{D'_i}\to\partial{\Sigma'}$ be gluing maps of the disks $D_i$ and $D_i$, respectively.
Then, the map $g'^{-1}\circ f\circ g_i:\partial{D_i}\to\partial{D'_i}$ is a diffeomorphism on a circle.
By the theorem of Smale~\cite{Sma59}, there exists a diffeomorphism $\phi_{i}:D_i\to D'_i$ such that $\phi_{i}\mid_{\partial{D_i}}=g'^{-1}_i\circ f\circ g_i$.
Then, define the diffeomophism $\widehat{f}:\widehat{\Sigma}\to\widehat{\Sigma'}$ by $\widehat{f}(x):=f(x)$ if $x\in\Sigma$ and $\widehat{f}(x):=\phi_{i}(x)$ if $x\in D_{i}$.
Note that $\widehat{f}$ is well-defined.
Then, $\widehat{f}(\widehat{\calD})=(\widehat{\Sigma'};f(\alpha),f(\beta),f(\gamma))$ becomes isotopic to $\widehat{\calD}'=(\widehat{\Sigma'};\alpha',\beta',\gamma')$ after performing the handleslides $s=(s_1,s_2,\ldots,s_n)$.
Here, we can consider that each handleslide $s_i$ is performed on $\widehat{\Sigma'}$, since $\Sigma'$ is a subsurface of $\widehat{\Sigma'}$.
\end{proof}
\begin{figure}[!htbp]
\centering
\includegraphics[scale=0.9]{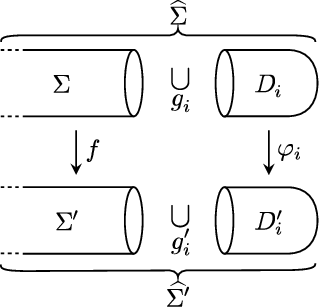}
\caption{A construction of $\widehat{f}:\widehat{\Sigma}\to\widehat{\Sigma'}$.}
\label{fig:fhat_gdthm}
\end{figure}

The following is an immediate corollary.

\begin{cor}\label{cor:equiv-rtd-X-D4}
Let $X$ be a closed $4$-manifold.
Let $\calD$ and $\calD'$ be two $(g,k;0,1)$-relative trisection diagrams of $X-\Int{D^4}$, and assume that they are diffeomorphism and handleslide equivalent.
Then, the $(g,k)$-trisection diagrams $\widehat{\calD}$ and $\widehat{\calD}'$ of the closed $4$-manifold $X$ are also diffeomorphism and handleslide equivalent.
\end{cor}

\section{A proof of the non-equivalence of relative trisections}\label{sect:proof_nonequiv}

In this section, we will prove the non-equivalence of relative trisection diagrams by using the operation introduced in the previous section.

\begin{lem}\label{lem:D1D2notdiffeo}
The two diagrams $\calD_1$ and $\calD_2$ of Figures~\ref{fig:td-S2xD2_2} are $(2,1;0,2)$-relative trisection diagrams of $S^2\times D^2$.
In addition, $\calD_1$ and $\calD_2$ are not diffeomorphism and handleslide equivalent.
\end{lem}

\begin{proof}
In \cite{KimMil20}, Kim and Miller gave an algorithm that produces a handlebody diagram of a $4$-manifold with boundary from a relative trisection diagram.
By using the algorithm, we see that $\calD_1$ and $\calD_2$ induce the handlebody diagrams of $S^2\times D^2$ shown in Figure~\ref{fig:Kd-S2xD2}.
Note that the dotted circles are induced from the cut arcs, and the attaching circles of $2$-handles are induced from $\gamma$-curves.

Assume that $\calD_1$ and $\calD_2$ are diffeomorphism and handleslide equivalent.
By Theorem~\ref{thm:equiv-rtd-gd}, the two $(2,0)$-trisection diagrams $\widehat{\calD_1}$ and $\widehat{\calD_2}$ shown in Figures~\ref{fig:td-S2xS2_CP2CP2bar} are also diffeomorphism and handleslide equivalent.
Thus, the two induced closed $4$-manifolds $X_{\widehat{\calD_1}}$ and $X_{\widehat{\calD_2}}$ are diffeomorphic.
However, $X_{\widehat{\calD_1}}\cong S^2\times S^2$ and $X_{\widehat{\calD_2}}\cong S^2\tilde{\times}S^2 \cong \CC{P^2}\#\overline{\CC{P^2}}$.
\end{proof}
\begin{figure}[!htbp]
\centering
\includegraphics[scale=1.1]{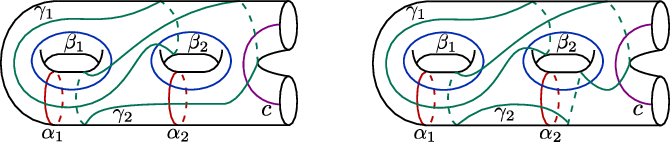}
\caption{Two $(2,1;0,2)$-relative trisection diagrams of $S^2\times D^2$. The left is $\calD_1$, and the right is $\calD_2$. The purple arc $c$ in each diagram is a cut system.}
\label{fig:td-S2xD2_2}
\end{figure}
\begin{figure}[!htbp]
\centering
\includegraphics[scale=0.85]{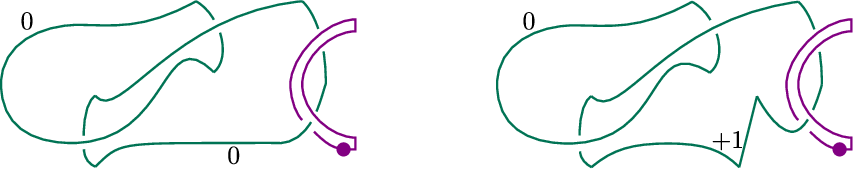}
\caption{Two handlebody diagrams of $S^2\times D^2$. The left diagram is induced from $\calD_1$, and the right one is induced from $\calD_2$.}
\label{fig:Kd-S2xD2}
\end{figure}
\begin{figure}[!htbp]
\centering
\includegraphics[scale=1.1]{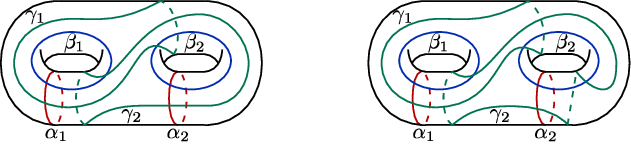}
\caption{Two $(2,0)$-trisection diagrams $\widehat{\calD_1}$ and $\widehat{\calD_2}$.}
\label{fig:td-S2xS2_CP2CP2bar}
\end{figure}

\begin{lem}\label{lem:S2xD2genus2}
The $4$-manifold $S^2\times D^2$ cannot admit a relative trisection of genus less than $2$.
\end{lem}

\begin{proof}
Assume that $S^2\times D^2$ admits a $(g,k;p,b)$-relative trisection of genus $g\in\{0,1\}$.
Let $A:=g+p+b-1-k$.
Since this value is the number of the pairs $(\delta_i, \epsilon_i)$ that $\delta_i$ and $\epsilon_i$ intersect at one point in Figure~\ref{fig:td-stdgkpb}, we see $0\leq A\leq g-p$.
It is known that, if a $4$-manifold $X$ admits a $(g,k;p,b)$-relative trisection, then the Euler characteristic $\chi(X)$ is equal to $g-3k+3p+2b-1$ (\cite[Corollary~2.10]{CasOzb19}).
Since $\chi(S^2\times D^2)=2$, we have $g-3k+3p+2b=3$.
By combining this condition and the definition of $A$, we obtain the followings:
\begin{align}
k &=-1-g+p+2A, \label{eqt:k} \\
b &= 3A-2g. \label{eqt:b}
\end{align}
Since $b\geq1$ and $0\leq A\leq g-p$, it holds that
\begin{align}
(2g+1)/3\leq A\leq g-p. \label{eqt:A}
\end{align}

If $g=0$, we see $p\leq-1/3$ by the inequality~(\ref{eqt:A}).
Since $p$ is a non-negative number, this is a contradiction.

If $g=1$, we have $1\leq A\leq 1-p$ by the inequality~(\ref{eqt:A}).
Hence, $p=0$ and $A=1$.
By using the conditions (\ref{eqt:k}) and (\ref{eqt:b}), we have $(g,k;p,b)=(1,0;0,1)$.
Recall that a $(g,k;p,b)$-relative trisection induces an open book decomposition on the boundary with page $\Sigma_{p,b}$.
Thus, we conclude that the boundary $S^2\times S^1$ admits an open book decomposition with disk pages.
This is a contradiction.
For a general discussion of the above, see Section~4 of \cite{Tak22a}.
\end{proof}

By combining Lemmas~\ref{lem:D1D2notdiffeo} and \ref{lem:S2xD2genus2}, we obtain the following, which implies Theorem~\ref{main:distinct-reltris}.

\begin{thm}\label{thm:2rtd_of_S2xD2}
Let $\calT_{\calD_1}$ and $\calT_{\calD_2}$ be two $(2,1;0,2)$-relative trisections of $S^2\times D^2$ induced by $\calD_1$ and $\calD_2$, resectively. Then, they satisfy the following conditions:
\begin{itemize}
\item
$\calT_{\calD_1}$ and $\calT_{\calD_2}$ are minimal genus relative trisections.
\item
$\calT_{\calD_1}$ and $\calT_{\calD_2}$ are not diffeomorphic.
\item
$\calT_{\calD_1}$ and $\calT_{\calD_2}$ induce the same open book decomposition $(\Sigma_{0,2}, \id)$ on the boundary $S^2\times S^1$.
\end{itemize}
\end{thm}

\begin{proof}[Proof of Theorem~\ref{thm:2rtd_of_S2xD2}]
Lemma~\ref{lem:S2xD2genus2} implies the minimality of $\calT_i$.
By Theorem~\ref{thm:rt-rtd} and Lemma~\ref{lem:D1D2notdiffeo}, the induced relative trisections $\calT_{\calD_1}$ and $\calT_{\calD_2}$ are not diffeomorphic.

The $(2,1;0,2)$-relative trisections $\calT_{\calD_1}$ and $\calT_{\calD_2}$ induce open book decompositions with page $\Sigma_{0,2}$.
If an open book with annulus page induces $S^2\times S^1$, then its monodromy is the identity map.
Thus, we conclude both $\calT_{\calD_1}$ and $\calT_{\calD_2}$ induce the same open book decomposition.
\end{proof}

\begin{rem}\label{rem:equiv-rtd-X-D4}
As the contrapositive of Corollary~\ref{cor:equiv-rtd-X-D4}, we see that, if a closed $4$-manifold $X$ admits two non-diffeomorphic $(g,k)$-trisections, then $X-\Int{D^4}$ also admits two non-diffeomorphic $(g,k;0,1)$-relative trisections.
\end{rem}

\subsection*{Acknowledgements}
The author would like to express his adviser Kouichi Yasui for helpful comments and encouragement.
The author was partially supported by JSPS KAKENHI Grant Number 24KJ1561.
%

\end{document}